\documentclass[oneside,english]{amsart}
\usepackage[T1]{fontenc}
\usepackage[latin9]{inputenc}
\usepackage{amsthm}
\usepackage{amstext}
\usepackage{amssymb}
\usepackage{esint}
\usepackage{float}
\restylefloat{table}

\makeatletter

\numberwithin{equation}{section}
\numberwithin{figure}{section}
\theoremstyle{plain}
\newtheorem{thm}{\protect\theoremname}[section]
  \theoremstyle{definition}
  \newtheorem{defn}[thm]{\protect\definitionname}
  \theoremstyle{plain}
  \newtheorem{prop}[thm]{\protect\propositionname}
  \theoremstyle{remark}
  \newtheorem{rem}[thm]{\protect\remarkname}
  \theoremstyle{plain}
  
  \theoremstyle{plain}
  \newtheorem{cor}[thm]{\protect\corollaryname}
  \theoremstyle{plain}
  \newtheorem{lem}[thm]{\protect\lemmaname}

\makeatletter
\@addtoreset{thm}{section}
\makeatother

\makeatother

\usepackage{babel}
  \providecommand{\corollaryname}{Corollary}
  \providecommand{\definitionname}{Definition}
  \providecommand{\factname}{Fact}
  \providecommand{\lemmaname}{Lemma}
  \providecommand{\propositionname}{Proposition}
  \providecommand{\remarkname}{Remark}
\providecommand{\theoremname}{Theorem}

\begin{document}

\title{Principal Bundles over Statistical Manifolds}

\thanks{This subject is supported by the National
Natural Science Foundations of China (No. 61179031, No. 10932002.)
}

\author[D. Li]{Didong Li}
\address{School of Mathematics and Statistics, Beijing Institute
              of Technology, Beijing 100081, P. R. China}
\email{lididong@gmail.com}

\author[H. Sun]{Huafei Sun}
\address{School of Mathematics and Statistics, Beijing Institute
              of Technology, Beijing 100081, P. R. China}
\email{huafeisun@bit.edu.cn}

\thanks{The second author is the corresponding author}

\author[C. Tao]{Chen Tao}
\address{School of Mathematics and Statistics, Beijing Institute
              of Technology, Beijing 100081, P. R. China}
\email{matheart@gmail.com}

\author[L. Jiu]{Lin JIu}
\address{Department of Mathematics, Tulane University, New Orleans, LA 70118, U.S.A. }
\email{ljiu@tulane.edu}
\begin{abstract}
In this paper, we introduce the concept of principal bundles on statistical
manifolds. After necessary preliminaries on information geometry and
principal bundles on manifolds, we study the $\alpha$-structure of
frame bundles over statistical manifolds with respect to $\alpha$-connections,
by giving geometric structures. The manifold of one-dimensional normal
distributions appears in the end as an application and a concrete
example.
\end{abstract}



\maketitle

\section{Introduction}

The recognition of fibre bundles took place in the period 1935-1940.
After the first definition by H.Whitney, the theory of fibre bundles
has been developed by many mathematicians such as H.Hopf, E.Stiefel
and N.Steenrod({[}6{]}). Nowadays, the theory of fibre bundles, especially
(differentiable) principal bundles, plays an important role in many
fields such as differential geometry, algebraic topology, etc. As
an extraordinary example, the proof of Gauss-Bonnet-Chern formula({[}4{]}),
which lays the foundation of global differential geometry, by S.S.Chern
through a global approach, involves principal bundles and connections
in the key step. In particular, since the concept of connections is
of great importance in differential geometry, hence connections on
principal bundles attract much attention. From then on, increasing
concerns have been focused on the theory of fibre bundles and connections
on principal bundles. Section 3 introduces basic results on principal
bundles and Section 4 concentrates on the corresponding geometric
stuctures.

Applying differential geometry on probability and statistics, S.Amari
initiated the theory of information geometry({[}1,2{]}) working over
statistical manifolds, which are manifolds consisting of probability
density functions. A series of concepts such as $\alpha$-connections,
dual connections and Fisher metrics are introduced and studied. Surprisingly,
Amari found that the sectional curvature of the manifold consisting
of normal distributions is $-\frac{1}{2}$, which not only implies
that the manifold is isometric to a hyperbolic space, but also makes
it an important example. From then on, the theory of information geometry
has been developed and applied in many other fields besides mathematics.
The backgroup of information geometry is included in Section 2.

After reviewing some basic concepts through Sections 2, 3 and 4, without
providing proofs since they can be found from lots of references such
as {[}1-3{]} and {[}5{]}, in Section 5, we give the $\alpha$-structure on frame
bundles, which are certainly principal bundles, over statistical manifolds,
in terms of Theorem 5.8 and Corollary 5.10. It turns
out that the $\alpha$-structures on frame bundles of statistical
manifolds are always easier to handle due to the linear structure
on the matrix Lie group $GL(n,\mathbb{R})$. In the end, Section 6
discusses the $\alpha$-structures over manifold of normal distributions
as both an application on concrete case and a verification of results
in Section 5.

\section{Information Geometry on Statistical Manifolds}

We call
\[
S:=\left\{ p\left(x;\theta\right)|\theta\in\Theta\right\}
\]
a statistical manifold if $x$ is a random variable in sample space
$X$ and $p\left(x;\theta\right)$ is the probability density function,
which satisfies certain regular conditions. Here, $\theta=\left(\theta_{1},\theta_{2},\dots,\theta_{n}\right)\in\Theta$
is an $n$-dimensional vector in some open subset $\Theta\subset\mathbb{R}^{n}$
and $\theta$ can be viewed as the coordinates on manifold $S$.
\begin{defn}
The Riemannian metric on statistical manifolds is defined by the Fisher
information matrix:
\[
g_{ij}(\theta):=E[(\partial_{i}l)(\partial_{j}l)]=\int(\partial_{i}l)(\partial_{j}l)p(x;\theta)dx,\ \ i,j=1,2,\ldots,n,
\]
where $E$ denotes the expectation, $\partial_{i}:=\frac{\partial}{\partial\theta_{i}}$,
and $l:=l(x;\theta)=\log p(x;\theta)$.
\end{defn}

\begin{defn}
A family of connections $\nabla^{(\alpha)}$ defined (by S.Amari)
as follows
\[
<\nabla_{A}^{(\alpha)}B,C>:=E[(ABl)(Cl)]+\frac{1-\alpha}{2}E[(Al)(Bl)(Cl)]
\]
are called $\alpha$-connections, where $A,B,C\in\mathfrak{X}(S)$,
$ABl=A(Bl)$, and $\alpha\in\mathbb{R}$ is the parameter. \end{defn}
\begin{rem}
$\nabla^{\left(\alpha\right)}$ is not usually compatible with the
metric but always torsion free. Any connection $\nabla$ is called torsion free if for any vector
fields $X$ and $Y$,
\[
\nabla_{X}Y-\nabla_{Y}X-\left[X,Y\right]=0,
\]
where $[\cdot,\cdot]$ denotes the Lie bracket.\end{rem}
\begin{thm}
If the Riemannian connection coefficients and $\alpha$-connection
coefficients are denoted by $\Gamma_{ijk}$ and $\Gamma_{ijk}^{(\alpha)}$,
respectively, then
\[
\Gamma_{ijk}^{(\alpha)}=\Gamma_{ijk}-\frac{\alpha}{2}T_{ijk},
\]
where $T_{ijk}:=E[(\partial_{i}l)(\partial_{j}l)(\partial_{k}l)]$.
Note that $\Gamma_{ijk}^{(0)}=\Gamma_{ijk}$.\end{thm}
\begin{defn}
The Riemannian curvature tensor of $\alpha$-connections is defined
by (using Einstein summation convention)
\[
R_{ijkl}^{(\alpha)}=(\partial_{j}\Gamma_{ik}^{(\alpha)s}-\partial_{i}\Gamma_{jk}^{(\alpha)s})+(\Gamma_{jtl}^{(\alpha)}\Gamma_{ik}^{(\alpha)t}-\Gamma_{itl}^{(\alpha)}\Gamma_{jk}^{(\alpha)t}),
\]
where $\Gamma_{jk}^{(\alpha)s}=\Gamma_{jki}^{(\alpha)}g^{is}$ and
$\left(g^{is}\right)$ is the inverse matrix of the metric matrix
$\left(g_{mn}\right)$.
\end{defn}

\begin{defn}
We call the statistical manifold $S$ $\alpha$-flat if $R_{ijkl}^{(\alpha)}=0$
holds in some open set, and the coordinates $\theta$ $\alpha$-affine
if $\Gamma_{ijk}^{(\alpha)}=0$ in some open set.
\end{defn}

\begin{defn}
A (piecewise) smooth curve $\gamma:\left[0,1\right]\rightarrow S$
on $S$ is called an $\alpha$-geodesic if
\[
\nabla_{\gamma'\left(t\right)}^{\left(\alpha\right)}\gamma'\left(t\right)=0.
\]

\end{defn}

\section{Principal Bundles}
\begin{defn}
Suppose that $P$, $M$, and $G$ are all smooth manifolds, where
$G$ is also a (right) Lie transformation group on $P$ and $\pi:P\rightarrow M$
is a smooth surjection. $(P,\pi,M,G)$ is called a principal (differentiable)
(fibre) bundle if the following are true.\\
(1) The action of $G$ on $P$ is free, i.e., if $ug=u$, $\forall u\in P$, then $g$ is the identity in $G$;\\
(2) $\pi^{-1}(\pi(p))=pG:=\{pg|g\in G\}$, $\forall p\in P$;\\
(3) $\forall x\in M$, there exist $U\in\mathcal{N}(x):=\{U|x\in U,\ U\text{\ is\ an\ open\ set\ in}\ M\}$
and a diffeomorphism $\Phi_{U}:\pi^{-1}(U)\rightarrow U\times G$,
where $\Phi_{U}$ has two components, i.e, $\Phi_{U}=(\pi,\phi_{U})$,
s.t. $\phi_{U}:\pi^{-1}(U)\rightarrow G$ satisfying
\[
\phi_{U}(pg)=\phi_{U}(p)g,\ p\in P,\ g\in G.
\]

\end{defn}
$G$ is called the structure group of principal bundle $P$ and the
pair $(\pi^{-1}(U),\ \Phi_{U})$ is called the local trivialization.
\begin{defn}
Suppose that $(P,\pi,M,G)$ is a principal bundle, $(\pi^{-1}(U),\ \Phi_{U})$
and $(\pi^{-1}(V),\ \Phi_{V})$ are two local trivializations.
\begin{eqnarray*}
g_{UV}:U\cap V & \rightarrow & G,\\
x & \mapsto & \phi_{U}(p)(\phi_{V}(p))^{-1},\, p\in\pi^{-1}(x)
\end{eqnarray*}
is called the transition function between $(\pi^{-1}(U),\ \Phi_{U})$
and $(\pi^{-1}(V),\ \Phi_{V})$.
\end{defn}

\begin{defn}
$(F(E),\widetilde{\pi},M,GL(r;\mathbb{R}))$ is called the frame bundle
associated with vector bundle $(E,\pi,M,\mathbb{R}^{r},GL(r;\mathbb{R}))$.
In particular, when $E=TM$, the tangent bundle of manifold $M$,
$F(M):=F(TM)$ is called the frame bundle of manifold $M$.
\end{defn}
Frame bundle is one of the most important types of principal bundles
because of its various useful structures. Some results hold only on
frame bundles rather than on general principal bundles. Amazingly,
the transition functions of frame bundle are quite nature: the Jacobi
matrix, as stated in the next theorem.
\begin{thm}
$(F(E),\widetilde{\pi},M,GL(r;\mathbb{R}))$ and $(E,\pi,M,\mathbb{R}^{r},GL(r;\mathbb{R}))$
have the same family of transition functions. In particular, the common
transition functions of $(F(M),\widetilde{\pi},M,GL(n;\mathbb{R}))$
and $(TM,\pi,M,\mathbb{R}^{n},GL(n;\mathbb{R}))$ are the Jacobian
matrix of coordinates transitions: $(g_{\alpha\beta}(x))_{ij}=(\frac{\partial x_{\alpha}^{i}}{\partial x_{\beta}^{j}})$.\end{thm}
\begin{defn}
Let $(P,\pi,M,G)$ be a principal bundle.
\[
V_{p}:=\ker\pi_{*}=\{X\in T_{p}P|\pi_{*}(X)=0\}
\]
 is called the vertical subspace of $T_{p}P$.
\end{defn}

\begin{defn}
For principal bundle $(P,\pi,M,G)$, $H\subset TP$ is called a connection
on $P$ if\\
(1) $T_{p}P=V_{p}\oplus H_{p},\ p\in P$;\\
(2) $(R_{g})_{*p}(H_{p})=H_{pg},\ p\in P,g\in G$;\\
(3) and $\forall X\in\mathfrak{X}(P)$, its projections to $V$ and
$H$: $v(X)$ and $h(X)$, are both smooth.
\end{defn}
In other words, a connection $H$ is a smooth decomposition of tangent
spaces on $P$: vertical subspace $V$ and horizontal subspace $H$,
where the latter is right-invariant.

\begin{defn}
Let $(P,\pi,M,G)$ be a principal bundle, and $\mathfrak{g}$ be the
Lie algebra of structure group $G$.
\begin{eqnarray*}
\tau:\mathfrak{g} & \rightarrow & \mathfrak{X}(P),\\
A & \mapsto & \tau(A),\,\tau(A)(p):=(R_{p})_{*e}(A)
\end{eqnarray*}
is called the fundamental vector field induced by $A$, where $R_{p}:G\rightarrow\pi^{-1}(\pi(p)),$
and $R_{p}(g):=R(p,g)=p\cdot g\in\pi^{-1}(\pi(p))$.
\end{defn}
Obviously, the set of all fundamental vector fields is a Lie algebra
isomorphic to $\mathfrak{g}$.
\begin{defn}
Let $(P,\pi,M,G)$ be a principal bundle, and $\mathfrak{g}$ be the
Lie algebra $G$. $\theta:\mathfrak{X}(G)\rightarrow\mathfrak{g}$,
defined by
\[
\theta(X_{g}):=L_{g*}^{-1}(X_{g}),\ X_{g}\in TgG
\]
is called the canonical 1-form on $G$. Furthermore, let $g_{\alpha\beta}$
be the transition functions,
\[
\theta_{\alpha\beta}:\mathfrak{X}(U_{\alpha}\cap U_{\beta})\rightarrow\mathfrak{g}
\]
is given by
\[
\theta_{\alpha\beta}:=g_{\alpha\beta}^{*}\theta,
\]
that is, $\theta_{\alpha\beta}$ is a $\mathfrak{g}$-valued-1-form
on $M$, defined as the pull-back $\mathfrak{g}$-valued-1-form
of $\theta$ on $G$ by $g_{\alpha\beta}$.\end{defn}
\begin{thm}
Suppose that $(P,\pi,M,G)$ is a principal bundle, and $\mathfrak{g}$
is the Lie algebra $G$. The following definitions of connections
are equivalent:\\
\underline{Definition 1}. A connection on $P$ is a smooth $M$-distribution
$H\subset TP$ s.t.\\
(1) $T_{p}P=V_{p}\oplus H_{p},\ p\in P,$\\
(2) $(R_{g})_{*}(H_{p})=H_{pg},\ p\in P,g\in G.$\\
\underline{Definition 2}. A connection on $P$ is a smooth $\mathfrak{g}$-valued-1-form
field $\omega$ on $P$ s.t.\\
(3) $\omega(\tau(A))=A,\ A\in\mathfrak{g,}$\\
(4) $R_{g}^{*}(\omega(X))=Ad_{g^{-1}}(\omega(X)),\ g\in G,X\in TP.$\\
\underline{Definition 3}. A connection on $P$ is a family of smooth
$\mathfrak{g}$-valued-1-form fields $\omega_{\alpha}$ on $U_{\alpha}$
s.t.\\
(5) $\omega_{\beta}(p)=Ad(g_{\alpha\beta}^{-1}(p))\circ\omega_{\alpha}(p)+\theta_{\alpha\beta}(p),\ p\in U_{\alpha}\cap U_{\beta}$. \end{thm}
\begin{defn}
Assume that $(P,\pi,M,G)$ is a principal bundle, $\mathfrak{g}$
is the Lie algebra of $G$, and $H$ is a connection on $P$. $\omega:\mathfrak{X}(P)\rightarrow\mathfrak{g},$
defined by
\[
\omega(X):=\sigma_{u*}^{-1}(v(X)),\ X\in T_{u}P
\]
is called the connection form of $(P,H)$. Here $\sigma_{u}:G\rightarrow uG$
is the left action of $G$ on $P$.
\end{defn}
It is easy to check that $\omega$ is vertical: $\omega(H)=0$. In
fact, if we have a $\mathfrak{g}$-valued-1-form $\omega$ satisfying
conditions (3) and (4) in Theorem 3.9, then $H:=\ker(\omega)$ is
a connection on $P$ with $\omega$ as its connection form, which
is also right-covariant.
\begin{cor}
Let $(E,\pi,M,\mathbb{R}^{r},GL(r;\mathbb{R}))$ and $(F(E),\widetilde{\pi},M,GL(r;\mathbb{R}))$
be a vector bundle and its associated frame bundle, respectively.
Then there exists a 1-1 correspondence between the connections on
$E$ and the connections on $F(E)$.
\end{cor}

\begin{cor}
There exists a 1-1 correspondence between the connections on $M$
and $F(M)$.\end{cor}
\begin{defn}
Let $\pi_{*b}:H_{b}\rightarrow T_{\pi(b)}M$. For any $X\in\mathfrak{X}(M)$,
there exists unique $\widetilde{X}=\pi_{*}^{-1}(X)\in\mathfrak{X}(P)$,
called the horizontal lift of $X$, s.t. $\pi_{*}(\widetilde{X})=X$.\end{defn}
\begin{thm}
A vector field on $P$ is right-invariant if and only if it is the
horizontal lift of some vector filed on $M$. \end{thm}
\begin{defn}
Let $\gamma:(-\epsilon,\epsilon)\rightarrow M$ be a smooth curve
on $M$. $\widetilde{\gamma}:(-\epsilon,\epsilon)\rightarrow P$ is
called the horizontal lift of $\gamma$ if
\[
\pi(\widetilde{\gamma}(t))=\gamma(t),\ \widetilde{\gamma}'(t)\in H_{\widetilde{\gamma}(t)},\ t\in(-\epsilon,\epsilon).
\]
 \end{defn}
\begin{thm}
Let $\gamma:(-\epsilon,\epsilon)\rightarrow M$ be a smooth curve
on $M$ with $\gamma(0)=p$. Then\\
(1) for any $b\in\pi^{-1}(p)$, there exists a unique horizontal lift
$\widetilde{\gamma}$ s.t. $\widetilde{\gamma}(0)=b$.\\
(2) Let $\widetilde{\gamma}_{1}$ be another smooth curve with $\widetilde{\gamma}_{1}(0)=bg$,
$g\in G$, then $\widetilde{\gamma}_{1}$ is also a horizontal lift
of $\gamma$ if and only if $\widetilde{\gamma}_{1}(t)=\widetilde{\gamma}(t)g$,
$t\in(-\epsilon,\epsilon).$
\end{thm}
Hence horizontal lift curve is unique when initial point is fixed.
Furthermore, all other horizontal lifts are just formed by right-translations.

\section{Geometry on Principal Bundles}
\begin{defn}
Denote by $(P,\pi,M,G,H,\omega)$ a principal bundle with connection
$H$ and connection form $\omega$.
\[
\Omega:=d\omega+\frac{1}{2}\omega\wedge\omega
\]
 is called the curvature form, where $\Omega$ is a $\mathfrak{g}$-valued-2-form
on $P$. \end{defn}
\begin{prop}
The second structure equation holds that
\[
\Omega=d\omega+\frac{1}{2}[\ \omega,\omega\ ]=d\omega\circ h.
\]
\end{prop}
\begin{defn}
Let $(F(M),\pi,M,GL(n;\mathbb{R}))$ be a frame bundle over $M$.
$\theta:T(F(M))\rightarrow\mathbb{R}^{n}$, defined by
\[
\theta(Y_{u}):=u^{-1}(\pi_{*}Y_{u}),\ Y_{u}\in T_{u}F(M)
\]
 is called the canonical 1-form on $F(M)$, where
\[
u:\mathbb{R}^{n}\rightarrow T_{\pi(u)}M,\ u(\xi):=u\xi,\xi\in\mathbb{R}^{n}.
\]

\end{defn}
In fact, the canonical 1-from $\theta$ can only be defined on frame
bundles.
\begin{defn}
For any $\xi\in\mathbb{R}^{n}$, $H(\xi):F(M)\rightarrow H$ s.t.
\[
H(\xi)_{u}:=\pi_{*}^{-1}(u\xi)
\]
 is called the fundamental horizontal vector field, where $\pi_{*}:H_{u}\rightarrow T_{\pi(u)}M$
is linear isomorphism.
\end{defn}
Fundamental horizontal vector fields and fundamental vertical vector
fields are horizontal and vertical, respectively, hence \textquotedbl{}orthogonal\textquotedbl{}
to each other. In addition, they form a basis of $T(F(M))$, which
implies that $T(F(M))$ is a trivial bundle, or parallelizable.
\begin{defn}
Let $(F(M),\pi,M,GL(n;\mathbb{R}),H,\omega)$ be a frame bundle with
connection $H$ and connection form $\omega$.
\[
\Theta:=d\theta\circ h
\]
 is called the torsion form on $F(M)$, where $\Theta$ is a $\mathbb{R}^{n}$-valued-2-form
on $F(M)$. \end{defn}
\begin{prop}
The first structure equation holds:
\[
\Theta=d\theta+\omega\wedge\theta.
\]

\end{prop}
In fact, the first and the second structure equations are similar
to the structure equations on a smooth manifold with a connection.
\begin{thm}
Suppose that $(F(M),\pi,M,GL(n;\mathbb{R}),H,\omega)$ is a frame
bundle with connection $H$ and connection form $\omega$. Then the
torsion form $\Theta$ and the curvature form $\Omega$ satisfy the
following equations
\[
d\Theta=\Omega\wedge\theta-\omega\wedge\Theta,
\]
and
\[
d\Omega=\Omega\wedge\omega,
\]
which are called the first and the second Bianchi idendities, respectively. \end{thm}
\begin{defn}
Denote by $(F(M),\pi,M,GL(n;\mathbb{R}),H,\Omega,\Theta)$ a principal
bundle with connection $H$, connection form $\omega$, curvature
form $\Omega$ and torsion form $\Theta$. For any $X,Y,Z\in T_{p}M,\ W\in\mathfrak{X}(M),\ u\in\pi^{-1}(p)$,
we have
\[
\nabla_{X}W:=u\widetilde{X}(\theta(\widetilde{W})),
\]

\[
T(X,Y):=u(\Theta(\widetilde{X},\widetilde{Y})),
\]
and
\[
R(X,Y)Z:=u(\Omega(\widetilde{X},\widetilde{Y})u^{-1}(Z)),
\]
where $\widetilde{X}$, $\widetilde{Y}$ and $\widetilde{W}$ are
the horizontal lifts of the vector fields $X$, $Y$ and $W$, respectively.
\end{defn}
The right sides of all formulae involving geometric structures on
frame bundle $F(M)$ are irrelevant to base manifold $M$, which means
that geometric structures on base manifold can be calculated on bundles.
The importance lies in that geometric structures on frame bundle are
often easier to handle.
\begin{thm}
Suppose that $(F(M),H,\Omega,\Theta)$ is a frame bundle over $(M,\nabla,T,R)$,
where $\nabla$ is the connection of $M$ induced by $H$. Also denote
$T$ and $R$ the torsion tensor and curvature tensor on $M$, respectively.
For any $X,Y,Z\in\mathfrak{X}(M)$, we have
\[
T(X,Y)=\nabla_{X}Y-\nabla_{Y}X-[X,Y],
\]
and
\[
R(X,Y)Z=\nabla_{X}\nabla_{Y}Z-\nabla_{Y}\nabla_{X}Z-\nabla_{[X,Y]}Z.
\]
\end{thm}
\begin{cor}
Let $(F(M),H,\Omega,\Theta)$ be the frame bundle over $(M,\nabla,T,R)$.
Then
\begin{eqnarray*}
\gamma:(-\epsilon,\epsilon)\rightarrow M\text{\ is\ a\ geodesic} & \Longleftrightarrow\nabla_{\gamma'}\gamma'=0\Longleftrightarrow & \widetilde{\gamma}'(\theta(\widetilde{\gamma}'))=0;\\
(M,\nabla)\text{\ is\ torsion\ free} & \Longleftrightarrow\,\,\,\,\, T=0\,\,\,\,\,\Longleftrightarrow & \Theta=0;\\
(M,\nabla)\text{\ is\ flat} & \Longleftrightarrow\,\,\,\,\, R=0\,\,\,\,\,\Longleftrightarrow & \Omega=0.
\end{eqnarray*}

\end{cor}
Based on these results, there are simple approaches to determine whether
a curve $\gamma$ on $(M,\nabla)$ is a geodesic, and whether $(M,\nabla)$
is torsion free, flat or not.

\section{$\alpha$-structure on Frame Bundles over Statistical Manifolds}

Throughout this section, we let $S=\{p(x;\theta|\theta\in\Theta)\}$
be an $n$-dimensional statistical manifold with coordinates charts
$\{(U_{\beta},x_{\beta}^{i})|\beta\in J\}$. Moreoever, Define $e_{i}^{\beta}:=\frac{\partial}{\partial x_{\beta}^{i}}$
and $\omega_{\beta}^{i}:=dx_{\beta}^{i}$ , which is the dual 1-form
of $e_{i}^{\beta}$ on $U_{\beta}$, $\forall1\leq i\leq n$. Then,
let $(\omega_{\beta})_{j}^{k}:=(\Gamma_{\beta})_{ji}^{k}\omega_{\beta}^{i}$
denote the connection form of the Riemannian connection $\nabla$.
\begin{defn}
The $\alpha$-connection form is defined by
\[
(\omega_{\beta}^{(\alpha)})_{j}^{k}:=(\Gamma_{\beta}^{(\alpha)})_{ji}^{k}\omega_{\beta}^{i},
\]
which is a $GL(n;\mathbb{R})$-valued-1-form on $U_{\beta}$. \end{defn}
\begin{rem}
The indices here are of different meanings. The super index $\alpha$
with parentheses is the same index with respect to $\alpha$-connection
$\nabla^{\left(\alpha\right)}$, while the lower index $\beta$ follows
from the index of coordinates $\left\{ \left(U_{\beta},x_{\beta}^{i}\right)\right\} $. \end{rem}
\begin{defn}
Let $F(S)$ be the frame bundle over $S$ with local trivialization
$\{(U_{\beta},\phi_{\beta},\Phi_{\beta})|\beta\in J\}$. Define
\[
\widetilde{\omega}_{\beta}^{(\alpha)}(u):=Ad(\phi_{\beta}^{-1})\circ\pi^{*}\omega_{\beta}^{(\alpha)}(u)+\phi_{\beta}^{*}\theta(u),\ u\in\pi^{-1}(U_{\beta}).
\]
Then, by Theorem 3.9, $\widetilde{\omega}^{(\alpha)}:=(\widetilde{\omega}_{\beta}^{(\alpha)})$
is a well defined $GL(n;\mathbb{R})$-valued-1-form globally on $F(S)$.
Hence, there exists a unique connection on $F(S)$ with $\widetilde{\omega}^{(\alpha)}$
as its connection form, which is denoted by $H^{(\alpha)}$. Now,
$(H^{(\alpha)},\widetilde{\omega}^{(\alpha)})$ is a family of connections
on the principal bundle $F(S)$.
\end{defn}
With such connections on $F(S)$, geometric structures can be defined
as that in Section 4.
\begin{defn}
Let $(F(s),H^{(\alpha)},\widetilde{\omega}^{(\alpha)})$ be the frame
bundle with respect to $\alpha$-connection over $S$.
\[
\Theta^{(\alpha)}:=d\theta\circ h^{(\alpha)}
\]
and
\[
\Omega^{(\alpha)}:=d\widetilde{\omega}^{(\alpha)}\circ h^{(\alpha)}
\]
are called the $\alpha$-torsion form and $\alpha$-curvature form
on $F(S)$, respectively.
\end{defn}

\begin{defn}
$\forall\xi\in\mathbb{R}^{n}$, a vector field $H^{(\alpha)}(\xi):\, F(S)\rightarrow H^{(\alpha)}$
defined by
\[
H^{(\alpha)}(\xi)_{u}:=\pi_{*}^{-1}(u\xi)
\]
is called the fundamental $\alpha$-horizontal vector field, where
$\pi_{*}:H_{u}^{(\alpha)}\rightarrow T_{\pi(u)}M$ is linear isomorphism.
\end{defn}
This definition is an analog to definition 4.4, corresponding to different
connections on frame bundles.

Properties of $GL\left(n;\mathbb{R}\right)$ and direct computation
give the following lemma
\begin{lem}
Denote by $(F(S),H^{(\alpha)},\widetilde{\omega}^{(\alpha)},\Theta^{(\alpha)},\Omega^{(\alpha)})$
the frame bundle with $\alpha$-connection, $(H^{(\alpha)},\widetilde{\omega}^{(\alpha)})$,
$\alpha$-torsion form $\Theta^{(\alpha)}$ and $\alpha$-curvature
form $\Omega^{(\alpha)}$. Then we have \\
(1) $\theta(H^{(\alpha)}(\xi))=\xi$, $\xi\in\mathbb{R}^{n};$\\
(2) $R_{g*}(H^{(\alpha)}(\xi)_{u})=H^{(\alpha)}(g^{-1}\xi)_{ug}$,
$g\in GL(n;\mathbb{R}^{n}),\ \xi\in\mathbb{R}^{n};$\\
(3) $[\tau(A),H^{(\alpha)}(\xi)]=H^{(\alpha)}(A\xi)$, $A\in gl(n;\mathbb{R}^{n}),\ \xi\in\mathbb{R}^{n}.$
\end{lem}

Also, following Proposition 4.6 and necessary computations, it is
not hard to obtain the next proposition.
\begin{prop}
Let $(F(S),H^{(\alpha)},\widetilde{\omega}^{(\alpha)},\Theta^{(\alpha)},\Omega^{(\alpha)})$
be the frame bundle over $S$. Then we have
\[
\Theta^{(\alpha)}=d\theta+\widetilde{\omega}^{(\alpha)}\wedge\theta,
\]
 and
\[
\Omega^{(\alpha)}=d\widetilde{\omega}^{(\alpha)}+\widetilde{\omega}^{(\alpha)}\wedge\widetilde{\omega}^{(\alpha)}.
\]

\end{prop}

Now, the main theorem follows.

\begin{thm}
$\forall X,Y,Z\in\mathfrak{X}(M)$, we have
\begin{equation}
\nabla_{X}^{(\alpha)}Y=u(\widetilde{X}^{(\alpha)}(\theta(\widetilde{Y}^{(\alpha)}))),
\end{equation}
\begin{equation}
T^{(\alpha)}(X,Y)=u(\Theta^{(\alpha)}(\widetilde{X},\widetilde{Y}))=\nabla_{X}^{(\alpha)}Y-\nabla_{Y}^{(\alpha)}X-[X,Y],
\end{equation}
and
\begin{equation}
R^{(\alpha)}(X,Y)Z=u(\Omega^{(\alpha)}(\widetilde{X},\widetilde{Y})u^{-1}(Z))=\nabla_{X}^{(\alpha)}\nabla_{Y}^{(\alpha)}Z-\nabla_{Y}^{(\alpha)}\nabla_{X}^{(\alpha)}Z-\nabla_{[X,Y]}^{(\alpha)}Z,
\end{equation}
 where $\widetilde{X}^{(\alpha)}$ is the horizontal lift of $X$
corresponding to connection $H^{(\alpha)}$. \end{thm}
\begin{proof}
Directly from Corollary 3.11 and Definition 4.8,
\[
\nabla_{X}^{(\alpha)}Y=u(\widetilde{X}^{(\alpha)}(\theta(\widetilde{Y}^{(\alpha)}))).
\]
For any $p\in M$, and $u\in\pi^{-1}(p)$, we get
\begin{align*}
T^{(\alpha)}(X,Y) & =u(\Theta^{(\alpha)}(\widetilde{X}_{u}^{(\alpha)},\widetilde{Y}_{u}^{(\alpha)}))\\
 & =u(d\theta(h^{(\alpha)}(\widetilde{X}_{u}^{(\alpha)}),h^{(\alpha)}(\widetilde{Y}_{u}^{(\alpha)})))\\
 & =u(d\theta(\widetilde{X}_{u}^{(\alpha)},\widetilde{Y}_{u}^{(\alpha)}))\\
 & =u(\widetilde{X}_{u}^{(\alpha)}(\theta(\widetilde{Y}^{(\alpha)})))-u(\widetilde{Y}_{u}^{(\alpha)}(\theta(\widetilde{X}^{(\alpha)})))-u(\theta([\widetilde{X}^{(\alpha)},\widetilde{Y}^{(\alpha)}]))\\
 & =\nabla_{X}^{(\alpha)}Y-\nabla_{Y}^{(\alpha)}X-[X,Y],
\end{align*}
and
\begin{align*}
\pi_{*}(H^{(\alpha)}(\widetilde{Y}^{(\alpha)}(\theta(\widetilde{Z}))_{u})) & =\pi_{*}(\pi_{*}^{-1}(u(\widetilde{Y}_{u}^{(\alpha)}(\theta(\widetilde{Z}^{(\alpha)})))))\\
 & =u(\widetilde{Y}_{u}^{(\alpha)}(\theta(\widetilde{Z}^{(\alpha)})))\\
 & =\nabla_{Y}^{(\alpha)}Z.
\end{align*}
 This implies that $H^{(\alpha)}(\widetilde{Y}^{(\alpha)}(\theta(\widetilde{Z}))_{u})$
is the horizontal lift of $W:=\nabla_{Y}^{(\alpha)}Z$, then,
\begin{align*}
\nabla_{X}^{(\alpha)}\nabla_{Y}^{(\alpha)}Z & =u(\widetilde{X}_{u}^{(\alpha)}(\theta(\widetilde{W}^{(\alpha)})))\\
 & =u(\widetilde{X}_{u}^{(\alpha)}(\theta(H^{(\alpha)}(\widetilde{Y}^{(\alpha)}(\theta(\widetilde{Z}^{(\alpha)}))))))\\
 & =u(\widetilde{X}_{u}(\widetilde{Y}(\theta(\widetilde{Z}^{(\alpha)})))).
\end{align*}
 By computing the right side of (5.3),
\begin{align}
 & \quad\nabla_{X}^{(\alpha)}\nabla_{Y}^{(\alpha)}Z-\nabla_{Y}^{(\alpha)}\nabla_{X}^{(\alpha)}Z-\nabla_{[X,Y]}^{(\alpha)}Z\notag\\
 & =u(\widetilde{X}_{u}^{\left(\alpha\right)}(\widetilde{Y}^{\left(\alpha\right)}(\theta(\widetilde{Z}^{(\alpha)}))))-u(\widetilde{Y}_{u}^{\left(\alpha\right)}(\widetilde{X}^{\left(\alpha\right)}(\theta(\widetilde{Z}^{(\alpha)}))))-u(\widetilde{[\widetilde{X}^{(\alpha)},\widetilde{Y}^{(\alpha)}]}^{(\alpha)}(\theta(\widetilde{Z}^{(\alpha)})))\\
 & =u((\widetilde{X}_{u}^{(\alpha)}(\widetilde{Y}^{(\alpha)})-\widetilde{Y}_{u}^{(\alpha)}(\widetilde{X}^{(\alpha)})-h([\widetilde{X}^{(\alpha)},\widetilde{Y}^{(\alpha)}]))\theta(\widetilde{Z}^{(\alpha)}))\\
 & =u(v([\widetilde{X}^{\left(\alpha\right)},\widetilde{Y}^{(\alpha)}])(\theta({\widetilde{Z}^{(\alpha)}}))),\notag
\end{align}
where $v([\widetilde{X}^{(\alpha)},\widetilde{Y}^{(\alpha)}])\in V$
indicates $\exists A\in gl(n;\mathbb{R})$ s.t. $\tau(A)=v([\widetilde{X}^{(\alpha)},\widetilde{Y}^{(\alpha)}])$.
\begin{align*}
v([\widetilde{X}^{(\alpha)},\widetilde{Y}^{(\alpha)}])(\theta(\widetilde{Z}^{(\alpha)})) & =\tau(A)(\theta(\widetilde{Z}^{(\alpha)}))\\
 & =\frac{d}{dt}\biggl{|}_{t=0}(\theta(\widetilde{Z}^{(\alpha)})(u\exp tA))\\
 & =\frac{d}{dt}\biggl{|}_{t=0}(\theta(\widetilde{Z}^{(\alpha)})\circ R_{\exp tA}(u))\\
 & =\frac{d}{dt}\biggl{|}_{t=0}((\exp tA)^{-1}\theta(\widetilde{Z}^{(\alpha)}(u)))\\
 & =-A\theta(\widetilde{Z}^{(\alpha)})(u).
\end{align*}
Computing the left side of (5.3) provides
\begin{align*}
R^{(\alpha)}(X,Y)Z & =u(\Omega^{(\alpha)}(\widetilde{X}_{u}^{(\alpha)},\widetilde{Y}_{u}^{(\alpha)})u^{-1}(Z))\\
 & =u(d\widetilde{\omega}^{(\alpha)}(\widetilde{X}_{u}^{(\alpha)},\widetilde{Y}_{u}^{(\alpha)})u^{-1}(Z))\\
 & =u(-\widetilde{\omega}^{(\alpha)}([\widetilde{X}^{(\alpha)},\widetilde{Y}^{(\alpha)}])(u)u^{-1}(Z))\\
 & =-u(\widetilde{\omega}\circ\tau(A)u^{-1}(Z))\\
 & =-u(A\theta(\widetilde{Z}_{u}^{(\alpha)})).
\end{align*}
Therefore we obtain
\[
R^{(\alpha)}(X,Y)Z=u(\Omega^{(\alpha)}(\widetilde{X},\widetilde{Y})u^{-1}(Z))
\]
as desired.
\end{proof}
In addition, the following lemma verifies the step from (5.4) to (5.5).
\begin{lem}
For any $X,Y\in\mathfrak{X}(M)$,
\[
\widetilde{[X,Y]}=h([\widetilde{X},\widetilde{Y}]).
\]
\end{lem}
\begin{proof}
Direct computation shows
\begin{align*}
\pi_{*}(h([\widetilde{X},\widetilde{Y}])) & =\pi_{*}(h([\widetilde{X},\widetilde{Y}]))+0\\
 & =\pi_{*}(h([\widetilde{X},\widetilde{Y}]))+\pi_{*}(v([\widetilde{X},\widetilde{Y}]))\\
 & =\pi_{*}(h([\widetilde{X},\widetilde{Y}])+v([\widetilde{X},\widetilde{Y}]))\\
 & =\pi_{*}([\widetilde{X},\widetilde{Y}])\\
 & =[\pi_{*}(\widetilde{X}),\pi_{*}(\widetilde{Y})]\\
 & =[X,Y].
\end{align*}

\end{proof}
As a consequence of the Theorem 5.8 and analog of Corollary 4.10,
the following corollary holds.
\begin{cor}
Let $(F(S),H^{(\alpha)},\Theta^{(\alpha)},\Omega^{(\alpha)})$ be
the frame bundle with $\alpha$-structure over $(S,\nabla^{(\alpha)},T^{(\alpha)},R^{(\alpha)})$.
\begin{eqnarray*}
\gamma:(-\epsilon,\epsilon)\rightarrow M\text{\ is\ an\ }\alpha\text{-geodesic} & \Longleftrightarrow\nabla_{\gamma'}^{(\alpha)}\gamma'=0\Longleftrightarrow & \widetilde{\gamma}'(\theta(\widetilde{\gamma}'^{(\alpha)}))=0;\\
(M,\nabla^{(\alpha)})\text{\ is\ torsion\ free} & \Longleftrightarrow\,\,\, T^{(\alpha)}=0\,\,\,\Longleftrightarrow & \Theta^{(\alpha)}=0;\\
(M,\nabla^{(\alpha)})\text{\ is\ flat} & \Longleftrightarrow\,\,\, R^{(\alpha)}=0\,\,\,\Longleftrightarrow & \Omega^{(\alpha)}=0.
\end{eqnarray*}

\end{cor}

\section{Application on Normal Distribution manifold.}

Let us focus on a concrete case: manifold of one-dimensional normal
distributions, defined by
\[
S:=\{p(x;\theta^{1},\theta^{2})|(\theta^{1},\theta^{2})\in\mathbb{R}\times\mathbb{R}_{+}\},
\]
where $\theta=(\theta^{1},\theta^{2})=(\mu,\sigma)$ are coordinates
and
\[
p(x;\theta^{1},\theta^{2})=p(x;\mu,\sigma)=\frac{1}{\sqrt{2\pi}\sigma}\exp\{-\frac{(x-\mu)^{2}}{2\sigma^{2}}\}
\]
is the probability density function of a one-dimensional normal distribution
with expectation $\mu$ and standard derivation $\sigma$. For the
geometric structures on it, directly we could compute that
\[
l(x,\theta)=\log p(x,\theta)=-\frac{(x-\mu)^{2}}{2\sigma^{2}}-\log(\sqrt{2\pi}\sigma),
\]
\[
\begin{cases}
\partial_{1}l=\frac{\partial}{\partial\mu}l=\frac{x-\mu}{\sigma^{2}},\\
\partial_{2}l=\frac{\partial}{\partial\sigma}l=\frac{(x-\mu)^{2}}{\sigma^{3}}-\frac{1}{\sigma},\\
\partial_{1}\partial_{1}l=-\frac{1}{\sigma^{2}},\\
\partial_{1}\partial_{2}l=-\frac{2(x-\mu)}{\sigma^{3}},\\
\partial_{2}\partial_{2}l=-\frac{3(x-\mu)^{2}}{\sigma^{4}}+\frac{1}{\sigma^{2}}.
\end{cases}
\]
and
\[
\begin{cases}
g_{11}=E[\partial_{1}l\partial_{1}l]=E[\frac{(x-\mu)^{2}}{\sigma^{4}}]=\frac{1}{\sigma^{2}},\\
g_{12}=g_{21}=E[\partial_{1}l\partial_{2}l]=E[\frac{x-\mu}{\sigma^{2}}(\frac{(x-\mu)^{2}}{\sigma^{3}}-\frac{1}{\sigma})]=0,\\
g_{22}=E[\partial_{2}l\partial_{2}l]=E[(\frac{(x-\mu)^{2}}{\sigma^{3}}-\frac{1}{\sigma})^{2}]=\frac{2}{\sigma^{2}}.
\end{cases}
\]
Namely, the Riemannian metric is given by
\[
g=(g_{ij})=\begin{bmatrix}\frac{1}{\sigma^{2}} & 0\\
0 & \frac{2}{\sigma^{2}}
\end{bmatrix}.
\]
Consequently, one can compute
\[
\begin{cases}
T_{111}=T_{122}=T_{212}=T_{221}=0,\\
T_{112}=T_{121}=T_{211}=\frac{2}{\sigma^{3}},\\
T_{222}=\frac{8}{\sigma^{3}}.
\end{cases}
\]
and
\[
\begin{cases}
\Gamma_{111}^{(\alpha)}=\Gamma_{122}^{(\alpha)}=\Gamma_{212}^{(\alpha)}=\Gamma_{221}^{(\alpha)}=0,\\
\Gamma_{112}^{(\alpha)}=\frac{1-\alpha}{\sigma^{3}},\\
\Gamma_{121}^{(\alpha)}=\Gamma_{211}^{(\alpha)}=-\frac{1+\alpha}{\sigma^{3}},\\
\Gamma_{222}^{(\alpha)}=-\frac{2+4\alpha}{\sigma^{3}}.
\end{cases}
\]
to obtain the curvature tensor
\[
R_{1212}^{(\alpha)}=\frac{1-\alpha^{2}}{\sigma^{4}}.
\]

After computing the main $\alpha$-structures on $S$ above, we switch
to the study of the same structures but in the \textquotedbl{}bundle
version\textquotedbl{}, which means all of them can be calculated
on the principal bundle (the frame bundle $F(S)$ over $S$). Before
the calculation, we make some necessary preparations.

The inverse of metric matrix $g$ is
\[
g^{-1}=(g^{ij})=\begin{bmatrix}\sigma^{2} & 0\\
0 & \frac{\sigma^{2}}{2}
\end{bmatrix}.
\]
Hence, coefficients of $\alpha$-connections are
\[
\begin{cases}
(\Gamma^{(\alpha)})_{11}^{1}=(\Gamma^{(\alpha)})_{12}^{2}=(\Gamma^{(\alpha)})_{21}^{2}=(\Gamma^{(\alpha)})_{22}^{1}=0,\\
(\Gamma^{(\alpha)})_{11}^{2}=\frac{1-\alpha}{2\sigma},\\
(\Gamma^{(\alpha)})_{12}^{1}=(\Gamma^{(\alpha)})_{21}^{1}=-\frac{1+\alpha}{\sigma},\\
(\Gamma^{(\alpha)})_{22}^{2}=-\frac{1+2\alpha}{\sigma}.
\end{cases}
\]
and then follows the connection form on $S$
\[
\omega^{(\alpha)}=((\omega^{(\alpha)})_{j}^{i})=\begin{bmatrix}-\frac{1+\alpha}{\sigma}d\theta^{2} & -\frac{1+\alpha}{\sigma}d\theta^{1}\\
\frac{1-\alpha}{2\sigma}d\theta^{1} & -\frac{1+2\alpha}{\sigma}d\theta^{2}
\end{bmatrix}.
\]
Now consider the frame bundle $F(S)$ over $S$. Since $S$ has a
global coordinates neighbourhood, the bundle is trivial, i.e., $F(S)=S\times GL(2;\mathbb{R})$.
For any $u\in F(S)$, $u$ represents a basis $(e_{1},e_{2})$ of
$T_{\pi(u)}S$. If $(e_{1},e_{2})=A(\frac{\partial}{\partial\theta^{1}},\frac{\partial}{\partial\theta^{2}})$,
the coordinates of $u$ is $u=(u^{1},\dots,u^{6})$, where
\[
u^{3}=A_{11},u^{4}=A_{12},u^{5}=A_{21},u^{6}=A_{22}.
\]
Furthermore, $\pi(u)=(\theta^{1},\theta^{2})=(u^{1},u^{2})$, hence
we have
\[
\pi_{*}(\frac{\partial}{\partial u^{1}}\biggl{|}_{u})=\frac{\partial}{\partial\theta^{1}}\biggl{|}_{\pi(u)},\pi_{*}(\frac{\partial}{\partial u^{2}})\biggl{|}_{u}=\frac{\partial}{\partial\theta^{2}}\biggl{|}_{\pi(u)}.
\]
And the local trivialization is
\begin{eqnarray*}
\Phi:F\left(S\right) & \rightarrow & S\times GL\left(2;\mathbb{R}\right),\\
u & \mapsto & (u^{1},u^{2},A)=(\theta^{1},\theta^{2},A).
\end{eqnarray*}
By Definition 5.3,
\[
\widetilde{\omega}^{(\alpha)}(u)=Ad(\phi^{-1}(u))\circ\pi^{*}\omega^{(\alpha)}+\phi^{*}\theta(u).
\]
Since $u=(u^{1},u^{2},\dots,u^{6})$ and $X_{u}=X^{i}\frac{\partial}{\partial u^{i}}\bigl{|}_{u}$,
we have
\begin{align*}
\widetilde{\omega}^{(\alpha)}(X_{u}) & =Ad(\phi^{-1}(u))\circ\pi^{*}\omega^{(\alpha)}(X_{u})+\phi^{*}\theta(X_{u})\\
 & =Ad(\phi^{-1}(u))\circ\omega^{(\alpha)}(\pi_{*}(X_{u}))+\theta(\phi_{*}(X_{u}))\\
 & =A\omega^{(\alpha)}(X^{1}\frac{\partial}{\partial\theta^{1}}\biggl{|}_{(\theta^{1},\theta^{2})}+X^{2}\frac{\partial}{\partial\theta^{2}}\biggl{|}_{(\theta^{1},\theta^{2})})A^{-1}+A^{-1}B\\
 & =A\begin{bmatrix}-\frac{1+\alpha}{\sigma}X^{2} & -\frac{1+\alpha}{\sigma}X^{1}\\
\frac{1-\alpha}{2\sigma}X^{1} & -\frac{1+2\alpha}{\sigma}X^{2}
\end{bmatrix}A^{-1}+A^{-1}B,
\end{align*}
where
\[
A=\begin{bmatrix}u^{3} & u^{4}\\
u^{5} & u^{6}
\end{bmatrix},\, B=\begin{bmatrix}X^{3} & X^{4}\\
X^{5} & X^{6}
\end{bmatrix}.
\]
Hence the horizontal space is
\[
H_{u}^{(\alpha)}=\ker(\widetilde{\omega}^{(\alpha)})=\left\{ X\in T_{u}F(S)\biggl{|}A\begin{bmatrix}-\frac{1+\alpha}{\sigma}X^{2} & -\frac{1+\alpha}{\sigma}X^{1}\\
\frac{1-\alpha}{2\sigma}X^{1} & -\frac{1+2\alpha}{\sigma}X^{2}
\end{bmatrix}A^{-1}+A^{-1}B=0\right\} .
\]

In particutlar, let $(e_{1},e_{2})=(\frac{\partial}{\partial\theta^{1}}\bigl{|}_{\pi(u)},\frac{\partial}{\partial\theta^{2}}\bigl{|}_{\pi(u)})$,
then $u=(u^{1},u^{2},1,0,0,1)$ and $\phi(u)=(\theta^{1},\theta^{2},I_{2\times2})$,
where $I_{2\times2}$ is the $2\times2$ identity matrix. Furthermore,
$\forall X\in T_{u}F(S)$, $X=X^{i}\frac{\partial}{\partial u^{i}}\bigl{|}_{u}$,
we have
\[
\pi_{*}(X)=X^{1}\frac{\partial}{\partial\theta^{1}}\biggl{|}_{(\theta^{1},\theta^{2})}+X^{2}\frac{\partial}{\partial\theta^{2}}\biggl{|}_{(\theta^{1},\theta^{2})}.
\]
So, the corresponding horizontal space becomes
\[
H_{u}^{(\alpha)}=\left\{ X_{u}\in T_{u}F(S)\biggl{|}\begin{bmatrix}-\frac{1+\alpha}{\sigma}X^{2} & -\frac{1+\alpha}{\sigma}X^{1}\\
\frac{1-\alpha}{2\sigma}X^{1} & -\frac{1+2\alpha}{\sigma}X^{2}
\end{bmatrix}+\begin{bmatrix}X^{3} & X^{4}\\
X^{5} & X^{6}
\end{bmatrix}=0\right\} .
\]
Under such circumstance, the horizontal projection of $X=X^{i}\frac{\partial}{\partial u^{i}}\in T_{u}F(S)$
is
\[
h^{(\alpha)}(X)=X^{1}\frac{\partial}{\partial u^{1}}+X^{2}\frac{\partial}{\partial u^{2}}+\frac{1+\alpha}{\sigma}X^{2}\frac{\partial}{\partial u^{3}}+\frac{1+\alpha}{\sigma}X^{1}\frac{\partial}{\partial u^{4}}-\frac{1-\alpha}{2\sigma}X^{1}\frac{\partial}{\partial u^{5}}+\frac{1+2\alpha}{\sigma}X^{2}\frac{\partial}{\partial u^{6}}.
\]
Then the horizontal lift of $X=X^{i}\frac{\partial}{\partial\theta^{i}}\in T_{(\theta^{1},\theta^{2})}S$
can also be expressed as
\[
\widetilde{X}_{u}=X^{1}\frac{\partial}{\partial u^{1}}+X^{2}\frac{\partial}{\partial u^{2}}+\frac{1+\alpha}{\sigma}X^{2}\frac{\partial}{\partial u^{3}}+\frac{1+\alpha}{\sigma}X^{1}\frac{\partial}{\partial u^{4}}-\frac{1-\alpha}{2\sigma}X^{1}\frac{\partial}{\partial u^{5}}+\frac{1+2\alpha}{\sigma}X^{2}\frac{\partial}{\partial u^{6}}.
\]
In particular, letting $X=\frac{\partial}{\partial\theta^{1}}$ and
$Y=\frac{\partial}{\partial\theta^{2}}$, we get
\[
\widetilde{X}=\frac{\partial}{\partial u^{1}}+\frac{1+\alpha}{\sigma}\frac{\partial}{\partial u^{4}}-\frac{1-\alpha}{2\sigma}\frac{\partial}{\partial u^{5}},
\]
and
\[
\widetilde{Y}=\frac{\partial}{\partial u^{2}}+\frac{1+\alpha}{\sigma}\frac{\partial}{\partial u^{3}}+\frac{1+2\alpha}{\sigma}\frac{\partial}{\partial u^{6}}.
\]
Thus,
\begin{equation}
\nabla_{X}^{(\alpha)}X=\nabla_{\frac{\partial}{\partial\theta^{1}}}^{(\alpha)}\frac{\partial}{\partial\theta^{1}}=(\Gamma^{(\alpha)})_{11}^{1}\frac{\partial}{\partial\theta^{1}}+(\Gamma^{(\alpha)})_{11}^{2}\frac{\partial}{\partial\theta^{2}}=\frac{1-\alpha}{2\sigma}\frac{\partial}{\partial\theta^{2}},
\end{equation}
\begin{equation}
\nabla_{X}^{(\alpha)}Y=\nabla_{\frac{\partial}{\partial\theta^{1}}}^{(\alpha)}\frac{\partial}{\partial\theta^{2}}=(\Gamma^{(\alpha)})_{12}^{1}\frac{\partial}{\partial\theta^{1}}+(\Gamma^{(\alpha)})_{12}^{2}\frac{\partial}{\partial\theta^{2}}=-\frac{1+\alpha}{\sigma}\frac{\partial}{\partial\theta^{1}},
\end{equation}
 and
\begin{equation}
\nabla_{Y}^{(\alpha)}Y=\nabla_{\frac{\partial}{\partial\theta^{2}}}^{(\alpha)}\frac{\partial}{\partial\theta^{2}}=(\Gamma^{(\alpha)})_{22}^{1}\frac{\partial}{\partial\theta^{1}}+(\Gamma^{(\alpha)})_{22}^{2}\frac{\partial}{\partial\theta^{2}}=-\frac{1+2\alpha}{\sigma}\frac{\partial}{\partial\theta^{2}}.
\end{equation}
On the other hand, consider
\[
\gamma_{1}(t):=(u^{1}+t,u^{2},1,0+\frac{1+\alpha}{\sigma}t,0-\frac{1-\alpha}{2\sigma}t,1)
\]
and
\[
\gamma_{2}(t):=(u^{1},u^{2}+t,1+\frac{1+\alpha}{\sigma}t,0,0,1+\frac{1+2\alpha}{\sigma}t).
\]
It is obvious that
\[
\gamma_{1}(0)=\gamma_{2}(0)=u,\gamma'_{1}(0)=\widetilde{X},\gamma'_{2}\left(0\right)=\widetilde{Y}.
\]
Thus,
\begin{align*}
u(\widetilde{X}_{u}(\theta(\widetilde{X}))) & =u\left(\frac{d}{dt}\biggl{|}_{t=0}\theta(\widetilde{X})\circ\gamma_{1}\right)\\
 & =u\left(\frac{d}{dt}\biggl{|}_{t=0}\left(\begin{bmatrix}1 & \frac{1+\alpha}{\sigma}t\\
-\frac{1-\alpha}{2\sigma}t & 1
\end{bmatrix}^{-1}\begin{bmatrix}1\\
0
\end{bmatrix}\right)\right)\\
 & =u\left(\frac{d}{dt}\biggl{|}_{t=0}\left(\frac{1}{1+\frac{1-\alpha^{2}}{2\sigma^{2}}t^{2}}\begin{bmatrix}1 & -\frac{1+\alpha}{\sigma}t\\
\frac{1-\alpha}{2\sigma}t & 1
\end{bmatrix}\begin{bmatrix}1\\
0
\end{bmatrix}\right)\right)\\
 & =u\left(\frac{d}{dt}\biggl{|}_{t=0}\frac{1}{1+\frac{1-\alpha^{2}}{2\sigma^{2}}t^{2}}\begin{bmatrix}1\\
\frac{1-\alpha}{2\sigma}t
\end{bmatrix}\right)\\
 & =u\left(\begin{bmatrix}0\\
\frac{1-\alpha}{2\sigma}
\end{bmatrix}\right)\\
 & =\frac{1-\alpha}{2\sigma}\frac{\partial}{\partial\theta^{2}}\\
 & =\nabla_{X}^{(\alpha)}X,\tag{6.4}
\end{align*}
\begin{align*}
u(\widetilde{X}_{u}(\theta(\widetilde{Y}))) & =u\left(\frac{d}{dt}\biggl{|}_{t=0}\theta(\widetilde{Y})\circ\gamma_{1}\right)\\
 & =u\left(\frac{d}{dt}\biggl{|}_{t=0}\left(\begin{bmatrix}1 & \frac{1+\alpha}{\sigma}t\\
-\frac{1-\alpha}{2\sigma}t & 1
\end{bmatrix}^{-1}\begin{bmatrix}0\\
1
\end{bmatrix}\right)\right)\\
 & =u\left(\frac{d}{dt}\biggl{|}_{t=0}\left(\frac{1}{1+\frac{1-\alpha^{2}}{2\sigma^{2}}t^{2}}\begin{bmatrix}1 & -\frac{1+\alpha}{\sigma}t\\
\frac{1-\alpha}{2\sigma}t & 1
\end{bmatrix}\begin{bmatrix}0\\
1
\end{bmatrix}\right)\right)\\
 & =u\left(\frac{d}{dt}\biggl{|}_{t=0}\frac{1}{1+\frac{1-\alpha^{2}}{2\sigma^{2}}t^{2}}\begin{bmatrix}-\frac{1+\alpha}{\sigma}t\\
1
\end{bmatrix}\right)\\
 & =u\left(\begin{bmatrix}-\frac{1+\alpha}{\sigma}\\
0
\end{bmatrix}\right)\\
 & =-\frac{1+\alpha}{\sigma}\frac{\partial}{\partial\theta^{1}}\\
 & =\nabla_{X}^{(\alpha)}Y,\tag{6.5}
\end{align*}
and
\begin{align*}
u(\widetilde{Y}_{u}(\theta(\widetilde{Y}))) & =u\left(\frac{d}{dt}\biggl{|}_{t=0}\theta(\widetilde{Y})\circ\gamma_{2}\right)\\
 & =u\left(\frac{d}{dt}\biggl{|}_{t=0}\left(\begin{bmatrix}1+\frac{1+\alpha}{\sigma}t & 0\\
0 & 1+\frac{1+2\alpha}{\sigma}t
\end{bmatrix}^{-1}\begin{bmatrix}0\\
\\
1
\end{bmatrix}\right)\right)\\
 & =u\left(\frac{d}{dt}\biggl{|}_{t=0}\begin{bmatrix}\frac{1}{1+\frac{1+\alpha}{\sigma}t} & 0\\
0 & \frac{1}{1+\frac{1+2\alpha}{\sigma}t}
\end{bmatrix}\begin{bmatrix}0\\
\\
1
\end{bmatrix}\right)\\
 & =u\left(\frac{d}{dt}\biggl{|}_{t=0}\begin{bmatrix}0\\
\frac{1}{1+\frac{1+2\alpha}{\sigma}t}
\end{bmatrix}\right)\\
 & =u\left(\begin{bmatrix}0\\
-\frac{1+2\alpha}{\sigma}
\end{bmatrix}\right)\\
 & =-\frac{1+2\alpha}{\sigma}\frac{\partial}{\partial\theta^{2}}\\
 & =\nabla_{Y}^{(\alpha)}Y.\tag{6.6}
\end{align*}
Formulae from (6.1) to (6.6) verify (5.1) of Theorem 5.8 directly.
Also, since both torsion $T^{(\alpha)}$ and curvature $R^{(\alpha)}$
can be derived by connection $\nabla^{(\alpha)}$, similarly one could
compute to check (5.2) and (5.3). In fact, we have
\[
g(u(\Omega^{(\alpha)}(\widetilde{X},\widetilde{Y})u^{-1}(Y)),Y)=\frac{1-\alpha^{2}}{\sigma^{4}}=R_{1212}^{(\alpha)}
\]
as desired.

These results show how our bundle approach simplifies the calculation
because all operations on the structure group $GL(2,\mathbb{R})$
are easier to handle. Since $GL(n,\mathbb{R})$ is a matrix Lie group,
the right translation and tangent mapping are actually products between
matrices, which are linear and therefore convenient to calculate.

\end{document}